\pdfoutput=1 % Required to make arXiv use pdflatex. Note that this breaks front page link.
\documentclass{amsart}
\usepackage[utf8]{inputenc}
\usepackage[T1]{fontenc}
\usepackage{lmodern}
\usepackage{graphicx}
\usepackage{longtable}
\usepackage{wrapfig}
\usepackage{rotating}
\usepackage[normalem]{ulem}
\usepackage{amsmath}
\usepackage{amssymb}
\usepackage{capt-of}
\usepackage{hyperref}
\usepackage{etoolbox}
\usepackage{microtype}
\usepackage{mathrsfs}
\usepackage[english]{babel}
\usepackage{tikz}
\usepackage{tikz-cd}
\usepackage{amsthm}
\usepackage{thmtools}
\usepackage{thm-restate}
\usepackage{xpatch}

\declaretheorem[numberwithin=section]{theorem}  % Theorem environments using amsthm + thmtools
\declaretheorem[sibling=theorem]{corollary}
\declaretheorem[sibling=theorem]{lemma}
\declaretheorem[sibling=theorem]{proposition}

\numberwithin{equation}{section}

\theoremstyle{definition} %How to get the example counter to be consistent with the theorem counter?
\newtheorem{example}{Example}[section]

\DeclareMathOperator\coker{coker}
\DeclareMathOperator\im{im}

\DeclareMathOperator\Hom{Hom}
\DeclareMathOperator\sHom{\mathcal{H}\mathit{om}}
\DeclareMathOperator\Ext{Ext}

\DeclareMathOperator\spec{Spec}

\DeclareMathOperator\supp{Supp}
\newcommand\dualab\hat
\newcommand\sh\mathscr
\newcommand\bb\mathbb

\DeclareMathOperator{\DR}{DR}

\DeclareMathOperator{\gr}{gr}
\renewcommand{\le}{\leqslant}
\renewcommand{\ge}{\geqslant}
\renewcommand{\subset}{\subseteq}
\DeclareMathOperator\End{End}
\DeclareMathOperator\Sym{Sym}

\makeatletter
\newcounter{proofstep}
\xpretocmd{\proof}{\setcounter{proofstep}{0}}{}{}
\newcommand{\proofstep}[1]{%
  \par
  \addvspace{\medskipamount}%
  \stepcounter{proofstep}%
  \noindent\emph{Step \theproofstep: #1}\par\nobreak\smallskip
  \@afterheading
}
\makeatother

\usepackage[backend=biber,style=alphabetic,giveninits=false,url=true,eprint=true,doi=false,isbn=false]{biblatex}
\renewbibmacro{in:}{%
\ifentrytype{article}{}{\printtext{\bibstring{in}\intitlepunct}}}
\author{Daniel Brogan}
\date{}
\title[Nori's Connectivity Theorem from the Perspective of $D$-Modules]{Nori's Connectivity Theorem from the Perspective of $D$-Modules}
\hypersetup{
 pdflang={English},
 final,
 bookmarks=true,
 bookmarksnumbered=true,
 bookmarksdepth=subsubsection}
\begin{document}

\maketitle
\begin{abstract}
    Given a very ample line bundle on a smooth projective variety, the variation of Hodge structure associated to the universal family of hyperplane sections can be thought of as a $D$-module with action generated by the Gauss-Manin connection. The decomposition theorem for a projective morphism and results similar to the Lefschetz hyperplane theorem tell us that the only nontrivial part of this VHS occurs in the middle degree corresponding to the vanishing cohomology of the hyperplane sections. We use Nori's connectivity theorem to give an explicit description of the cohomology sheaves of the de Rham complex of this $D$-module in terms of the vanishing cohomology of the original variety.
\end{abstract}

\section{Introduction}\label{Intro}
The goal of this paper is to understand the main result of \cite{Nor93} in a different light, namely we will interpret this theorem in the setting of $D$-modules. If $\mathcal M$ denotes the mixed Hodge module associated to a variation of Hodge structure of hyperplane sections of a smooth projective variety, then using Nori's theorem and the decomposition theorem for mixed Hodge modules we are able to get an explicit description of the cohomology sheaves of the graded pieces of $\DR(\mathcal M)$ in terms of the cohomology of the original variety.

%Outline of the paper?
\iffalse
The rest of the Introduction will recall the main result of Nori's paper and fix some notation. Section \ref{NoriPf} is devoted to the proof of Nori's main theorem Proposition \ref{prop:N2}. In Section \ref{DMod}, we will cover generalities about $D$-modules and state useful theorems such as the Riemann-Hilbert correspondence. This will allow us in Section \ref{MainThm} to reinterpret Proposition \ref{prop:N2} in the language of $D$-modules.\\
\fi

\indent For the entirety of this paper let $X$ be a smooth projective variety of dimension $n+1$ and $\mathcal O_X(1)$ a very ample line bundle on $X.$ Set $P=\mathbb{P}\left(H^0(X,\mathcal O_X(m))\right)$ where $m\gg 0;$ we take this to be the projective space of lines in $H^0(X,\mathcal O_X(m).$ Let $d=\dim P$ and let $\ell:\mathscr X\hookrightarrow P\times X$ be the universal hypersurface, i.e.
\[
	\mathscr X = \left\{([\sigma],x)\in P\times X \mid \sigma(x)=0\right\}.
\]
Let $pr_P$ and $pr_X$ denote the projections from $P\times X$ to $P$ and $X$ respectively. Write $\pi_P=pr_P|_{\mathscr X}$ and $\pi_X=pr_X|_{\mathscr X}.$ Using $\pi_P,$ a point $p\in P$ corresponds to a hypersurface $X_p=\pi_X(\pi_P^{-1}(p))\subseteq X.$ We will denote by $P^\mathrm{sm}$ the set of $p\in P$ for which $X_p$ is smooth and $\mathscr X^\mathrm{sm}$ the union of these $X_p.$ The main theorem that Nori proved is the following:
\begin{theorem}
	Suppose that $m\gg 0$ is sufficiently large Then the restriction map $\ell^*: H^k(P^\mathrm{sm}\times X,\mathbb Q)\to H^k(\mathscr X^\mathrm{sm},\mathbb Q)$ is an isomorphism for $k<2n$ and is injective for $k=2n.$
	\label{thm:N}
\end{theorem}
This echos the famous Lefschetz hyperplane theorem, which says that for general $p\in P,$ the restriction map $H^k(X,\mathbb Z)\to H^k(X_p,\mathbb Z)$ is an isomorphism for $k<n$ and injective for $k=n.$ The primary difference between Nori's theorem and the Lefschetz theorem is the need to consider sufficiently ample hypersurfaces. The ``$m\gg 0$" condition appearing here comes from applying Serre's vanishing theorem for certain auxillary coherent sheaves in the proof. Because of this, the size of $m$ necessary to get the result of the proposition is inexplicit. The main ingredient Nori used for proving Theorem \ref{thm:N} is the following proposition.

\begin{proposition}
	Fix a natural number $c$ and suppose that $m\gg 0$ is sufficiently large. Then for $a<n$ and $a+b<n+c$ the restriction map $$\ell^*:R^a (pr_P)_*\Omega^b_{P\times X}\to R^a (\pi_P)_*\Omega^b_{\mathscr X}$$ is an isomorphism; for $a\le n$ and $a+b\le n+c$ the map is injective.
\label{prop:N2}
\end{proposition}
\indent The sheaves appearing in Proposition \ref{prop:N2} are the cohomology of the graded pieces of the de Rham complexes of certain $D$-modules on $P,$ namely $(pr_P)_+\mathcal O_{P\times X}$ and $(\pi_P)_+\mathcal O_{\mathscr X}.$ These are the $D$-modules associated to the variations of Hodge structure given by the cohomology of $X$ and the cohomology of the hypersurfaces $X_p$ respectively. So Proposition \ref{prop:N2} can be restated to say that for certain graded pieces of the de Rham complex and certain low degrees, the restriction map 
\begin{equation*}
    \DR_P((pr_P)_+\mathcal O_{P\times X})\to\DR_P((\pi_P)_+\mathcal O_{\mathscr X})
\end{equation*}
induces an isomorphism on cohomology sheaves. Furthermore, the Lefschetz hyperplane theorem tells us the only nonconstant part of $(\pi_P)_+\mathcal O_{\mathscr X}$ is the $D$-module $\mathcal M$ corresponding to the vanishing cohomology of the $X_p.$ Explicitly writing the above restriction map using the decomposition theorem, we can find exactly which sheaves make up the cohomology of the graded pieces of $\DR_P(\mathcal M).$ We then get the main result of the paper.
\begin{proposition}
    Let $k$ and $b$ be integers such that $k<2n$ and $k-b<n.$ Then the restriction map $\ell^*:\DR_P((pr_P)_+\mathcal O_{P\times X})\to\DR_P((\pi_P)_+\mathcal O_{\mathscr X})[1]$ induces an isomorphism
    \begin{equation*}
	    \bigoplus_{j=n+1}^{k} L^{j-n-1}H_{\mathrm{prim}}^{b-k+n+1,k-b-j+n+1}(X)\otimes \Omega_P^{k-j} \cong \mathcal H^{-d-n+k}\gr_{-b}^F \DR(\mathcal M).
    \end{equation*}
    %Maybe I should reindex i=j-n-1 here...
\end{proposition}
In particular, this isomorphism holds over the smooth locus $P^\mathrm{sm}.$

\section{The Proof of Nori's Theorem}\label{NoriPf}
In this section we present Nori's proof of Proposition \ref{prop:N2} which can be found in section 3 of \cite{Nor93}. We will cover only the case in which $X_p$ is a hypersurface in $X$ for each $p\in P,$ however Nori treats the case of complete intersections with little additional work. We begin with a useful general statement. Given two smooth varieties $A$ and $B$ with $\ell:B\hookrightarrow A$ a closed immersion, we let $\Omega^k_{A,B}$ denote the kernel of the restriction morphism $\Omega^k_A\to\ell_*\Omega^k_B$ for each $k;$ we have an exact sequence
\[
	0\to\Omega^k_{A,B}\to\Omega_{A}^k\to\ell_*\Omega_{B}^k\to 0.
\]
Since our goal is to get an isomorphism between the final two terms after pushing forward by a smooth morphism, we should first say something about the vanishing of the first term.
\begin{lemma}
	Let $(A,B,S)$ be a triple as where $pr:A\to S$ and $\pi:B\to S$ are morphisms and $\ell:B\hookrightarrow A$ is a closed immersion. Let $g:T\to S$ be a smooth morphism. Write $A_T=A\times_S T$ and $B_T=B\times_S T.$ If the triple $(A,B,S)$ satisfies the condition
	\begin{equation}
	    R^a (pr)_*\Omega^b_{A,B} = 0 \text{ for $a\le n$ and $a+b\le n+c$}.
	    \label{eqn1:lem:1}
	\end{equation}
	for the natural numbers $n$ and $c,$ then so does the triple $(A_T,B_T,T).$ The converse is true if $g:T\to S$ is smooth and surjective.
\label{lem:1}
\end{lemma}
\begin{proof}
	First we prove the lemma in the case where $S$ and $T$ are affine and then pass to the general case. We have several morphisms induced by pulling back:
	\[
		pr_A:A_T\to A,\quad pr_T:A_T\to T,\quad \pi_B:B_T\to B,\quad \pi_T:B_T\to T
	\]
	along with the closed immersion $\ell_T:B_T\to A_T.$
	\[
	    % https://tikzcd.yichuanshen.de/#N4Igdg9gJgpgziAXAbVABwnAlgFyxMJZAJgBoBGAXVJADcBDAGwFcYkQAhEAX1PU1z5CKAMwVqdJq3YBBHnxAZseAkTHEJDFm0QgAyvP7KhRcuJpbpugCqHFAlcORmADJqk6QMgPq3eRwVUUF1I3Cw92Dl8eCRgoAHN4IlAAMwAnCABbJDMQHAgkMhA4AAssFJwkEMltdjQ0u3Ssqpp8nNb6LEZ2EogIAGsQcNrdAB1RmEZumkZ6ACNJgAUHE100rHiSyv8QJuzEMTyCxCLLT3ihkFmFxmXjIJB1ze2FPaQAFlbj3LO6tO85DN5ksVg9GDAKo0MvsAKxfFo1KwgcZoLDeLhAm53QLCR4bLZQ5qIOFHD4dLo9PqDYZI8aTRjRHZvRCfUkHGmeerRTEg+648GQpnQpAktrsxGeFFo2w826g3FPAlConVMWnCJjUaomLcIA
    \begin{tikzcd}
        B_T \arrow[rrd, "\pi_B"'] \arrow[r, "\ell_T", hook] \arrow[rd, "\pi_T"'] & A_T \arrow[rrd, "pr_A"] \arrow[d, "pr_T"] & & \\
        & T \arrow[rrd, "g"'] & B \arrow[r, "\ell"', hook] \arrow[rd, "\pi"] & A \arrow[d, "pr"]\\
        & & & S                
    \end{tikzcd}
	\]
	
	Since $g:T\to S$ is a smooth morphism, so too are $pr_A$ and $\pi_B.$ We also have that
	\[
		\Omega_{A_T/A}^1=pr_T^*\Omega_{T/S}^1\qquad\text{and}\qquad\Omega_{B_T/B}^1=\pi_T^*\Omega_{T/S}^1.
	\]
    Since $S$ and $T$ are affine we get a split exact sequence
	\[
		0\to g^*\Omega_S^1\to\Omega_T^1\to\Omega_{T/S}^1\to 0
	\]
	which by pulling back induces a splitting of the morphism of exact sequences
	%https://tikzcd.yichuanshen.de/#N4Igdg9gJgpgziAXAbVABwnAlgFyxMJZABgBpiBdUkANwEMAbAVxiRGJAF9T1Nd9CKAIzkqtRizZoATgH0AggD0AVAB1VAeQC2MAOZ0FioVx4gM2PASIAmUdXrNWiEOu16DwebIAqnIyd4LASIAZjtxRyk5bxVXHX1ZYG8AegBlP2NuQP4rFAAWcIdJZw4ssz5LQWQRITEipxAACnUYBgYfAEpZNVU0LFkAIVjNeIMhzNNzHKrbWvsJBubVVvbvLp63BOABnwyA8qDc5DC5iOKmlrbO7vU+n2HNjxT0-zKpyqIyU-q2UsmK4L5UjfBa-LhiGBQXTwIigABm0ggWiQZBAOAgSCEZQRSMx1HRSGs2MRyMQtjRGMQIWJuKp+MpeRppJEFKQAFYmYT6UgAGycumsxAAdn5AA5uYgOaYcaS2RK+dKSbyJSLFbShRKAJycCicIA
	\[
	\begin{tikzcd}
		0 \arrow[r] & pr_A^*\Omega_A^1 \arrow[r] \arrow[d]  & \Omega_{A_T}^1 \arrow[r] \arrow[d] & pr_T^*\Omega_{T/S}^1 \arrow[r] \arrow[d]  & 0 \\
		0 \arrow[r] & (\ell_T)_*\pi_B^*\Omega_B^1 \arrow[r] & (\ell_T)_*\Omega_{B_T}^1 \arrow[r] & (\ell_T)_*\pi_T^*\Omega_{T/S}^1 \arrow[r] & 0.
	\end{tikzcd}
	\]
	This means that both horizontal sequences split and the middle vertical morphism can be taken to respect the splitting. Taking exterior powers yields a commutative diagram,
	%https://tikzcd.yichuanshen.de/#N4Igdg9gJgpgziAXAbVABwnAlgFyxMJZABgBpiBdUkANwEMAbAVxiRAB12B5AWxgHM6AfWABBIQBUAvgD0ARiCml0mXPkIoAjOSq1GLNpzlZ+ENMzgi0AagCOAXjlS0AJyGiZAKk68Bwj2icEHh8cAAErpJePnyCIhIA9ADKsraKyiAY2HgERGSauvTMrIggABScMAwMkgCUQt7cscLAAEKSsgpKKtnqRNoF1EUGpUYmZhZWdo7Obh6NVTWNvnGtMoHswVihEW4S0U1+8cmpirowUPzwRKAAZi4QPEhkIDgQSNogDFhgJSBwEG+UBA1AAFjA6MDEGAmNVqDg6FgGGxIL8QV86HIqgAFVQ5DRfGC3HDpO4PJ6IABM8PeiAAzN0QPdHs8aUhKYzmRTPm8kAyKFIgA
	\[
	\begin{tikzcd}
		\Omega_{A_T}^b \arrow[r] \arrow[d] & \bigoplus_{p+q=b}pr_A^*\Omega_A^p\otimes pr_T^*\Omega_{T/S}^q \arrow[d] \\
		(\ell_T)_*\Omega_{B_T}^b \arrow[r] & \bigoplus_{p+q=b}pr_A^*\ell_*\Omega_B^p\otimes pr_T^*\Omega_{T/S}^q    
	\end{tikzcd},
	\]
	in which the horizontal arrows are isomorphisms. Hence we get an isomorphism
	\[
		\Omega_{A_T,B_T}^b\cong \bigoplus_{p+q=b}pr_A^*\Omega_{A,B}^p\otimes pr_T^*\Omega_{T/S}^q.    
	\]
	Thus
	\begin{equation}
		g_*R^a(pr_T)_*\Omega_{A_T,B_T}^b\cong \bigoplus_{p+q=b}R^a(pr_S)_*\Omega_{A,B}^p\otimes g_*\Omega_{T/S}^q.  
		\label{eqn2:lem:1}
	\end{equation}
	\indent Now assume that $(A,B,S)$ satisfies the condition of (\ref{eqn1:lem:1}) and fix $a$ and $b$ such that $a\le n$ and $a+b\le n+c.$ Then for each $b'\le b$ and each $t\in T$ we have $R^a(pr_S)_*\Omega_{A,B}^{b'}=0.$ In particular this sheaf vanishes at $g(t).$ It follows from (\ref{eqn2:lem:1}) that $R^a(pr_T)_*\Omega_{A_T,B_T}^{b}$ vanishes at $t.$ Thus $R^a(pr_T)_*\Omega_{A_T,B_T}^{b}=0$ and $(A_T,B_T,T)$ satisfies the condition of (\ref{eqn1:lem:1}).\\
	\indent Conversely, suppose that $g$ is surjective and $(A_T,B_T,T)$ satisfies the condition of (\ref{eqn1:lem:1}). Fix $a$ and $b$ such that $a\le n$ and $a+b\le n+c$ and $s\in S.$ Then since $R^a(pr_T)_*\Omega_{A_T,B_T}^{b}$ vanishes at each $t\in g^{-1}(s),$ (\ref{eqn2:lem:1}) tells us that $R^a(pr_S)_*\Omega_{A,B}^{b}$ vanishes at $s.$ Thus $(A,B,S)$ satisfies the condition of (\ref{eqn1:lem:1}) as well.
	For the general case  we pass to an appropriate affine cover of $S$ and $T$ and we apply the above argument.
\end{proof}

Now we return to the case of the universal hypersurface $\mathscr X$ lying in $P\times X$. We have a short exact sequence
\[
	0\to\Omega^b_{P\times X,\mathscr X}\to\Omega_{P\times X}^b\to\ell_*\Omega_{\mathscr X}^b\to 0.
\]
By passing to the long exact sequence associated to $(pr_X)_*$ we find that the conclusion of Proposition \ref{prop:N2} is equivalent to the vanishing of $R^a (pr_X)_*\Omega^b_{P\times X,\mathscr X}$ for $a\le n$ and $a+b\le n+c.$ Hence it suffices to prove
\begin{proposition}
	Suppose that $m\gg 0$ is sufficiently large. Then for $a\le n$ and $a+b\le n+c$ we have that
	\begin{equation}
		R^a (pr_X)_*\Omega^b_{P\times X,\mathscr X} = 0.
		\label{eqn:prop:N2'}
	\end{equation}
\label{prop:N2'}
\end{proposition}
If we have a triple $(A,B,S)$ where $pr:A\to S$ and $\pi:B\to S$ are morphisms and $B\hookrightarrow A$ is a closed immersion then we say that the triple $(A,B,S)$ satisfies the condition in(\ref{eqn:prop:N2'}) if for $a\le n$ and $a+b\le n+c$ we have that
\[
	R^a (pr)_*\Omega^b_{A,B} = 0.
\]
So the proposition states that the triple $(P\times X,\mathscr X, P)$ satisfies the condition in (\ref{eqn:prop:N2'}).

Let $S=H^0(X,\mathcal O_X(m)),$ $A=S\times X$ and $B\subseteq A$ the vanishing locus of the universal section. We claim it is enough to prove that the condition in (\ref{eqn1:lem:1}) holds for this triple $(A,B,S).$ Assuming that $(A,B,S)$ indeed satisfies the condition in (\ref{eqn1:lem:1}), let $U=S-\{0\}.$ Letting $A_U=U\times X$ and $B_U=A_U\cap B,$ we see that applying the lemma to the smooth map $U\to S$ gives us that $(A_U,B_U,U)$ satisfies the condition in (\ref{eqn1:lem:1}) as well. Finally we have a surjective smooth map $U\to P$ which after applying the lemma again tells us that $(P\times X,\mathscr X,P)$ satisfies the condition in (\ref{eqn1:lem:1}). Hence we aim to prove that $(A,B,S)$ satisfies the condition in (\ref{eqn1:lem:1}). For the rest of the proof we let	\[
	pr_X:A\to X,\quad pr_S:A\to S,\quad \pi_X:B\to X,\quad \pi_S:B\to S
\]
denote the projections and $\ell:B\to A$ the inclusion. Note that $pr_X$ and $\pi_X$ are vector bundles over $X$ of ranks $\dim H^0(X,\mathcal O_X(m))$ and $\dim H^0(X,\mathcal O_X(m))-1$ respectively.\\
\indent The map $pr_X$ is smooth, hence yields a Leray filtration $L$ on each sheaf of relative forms $\Omega_{A,B}$ given by the formula
\[
    L^m\Omega_{A,B}^b =\Omega_{A,B}^{b-m}\otimes pr_X^*\Omega_X^m.
\]
The graded pieces are
\[
\gr_L^m\Omega_{A,B}^b = \Omega_{(A,B)/X}^{b-m}\otimes pr_X^*\Omega_X^m
\]
where $\Omega^k_{(A,B)/X}=\Omega^k_{(A,B)}/pr_X^*\Omega_X.$ To prove the claim, it then suffices to prove it for each graded piece:
\begin{equation}\tag{A}
	\alpha \le n,\quad \alpha+\beta+\gamma \le n+c\implies R^\alpha(pr_S)_*\left(\Omega_{(A,B)/X}^{\gamma}\otimes pr_X^*\Omega_X^\beta\right)=0.
\label{eqn:A}
\end{equation}
\indent We can reduce the problem further. Let $\mathcal A=\Gamma(X,\mathcal O_X(m))\otimes_\mathbb C \mathcal O_X$ and $\mathcal B$ be the sheaves of sections of the vector bundles $A\to X$ and $B\to X$ respectively. They fit into the short exact sequence $$0\to\mathcal B\to\mathcal A\to\mathcal O_X(m)\to 0.$$ The sheaf of ideals for $B$ inside of $\mathcal O_A$ is $\mathcal I=pr_X^*\mathcal O_X(m).$ The map $\Omega_{A/X}^\gamma\to\Omega_{B/X}^\gamma$ can then be identified with the composition
\[
	pr_X^*\bigwedge\nolimits^{\!\gamma}\mathcal A^*\to pr_X^*\bigwedge\nolimits^{\!\gamma}\mathcal B^*\to pr_X^*\bigwedge\nolimits^{\!\gamma}\mathcal B^*\otimes \mathcal O_A/\mathcal I.
\]
If we let $\mathcal E^\gamma=\coker(\bigwedge^\gamma\mathcal B\to\bigwedge^\gamma\mathcal A),$ then it follows that the sequence
\[
	0\to pr_X^*\left(\mathcal E^\gamma\right)^*\to \Omega_{(A,B)/X}^\gamma\to pr_X^*\left(\bigwedge\nolimits^{\!\gamma}\mathcal B^*\right)\otimes\mathcal I\to 0
\]
is exact. Now observe that $\mathcal E^\gamma$ is a summand of $\mathcal D^\gamma=\coker(\mathcal B^{\otimes\gamma}\to\mathcal A^{\otimes\gamma})$ and that $\bigwedge^\gamma \mathcal B^*$ is a summand of $\left(\mathcal B^{\otimes\gamma}\right)^*.$ Putting this all together, (\ref{eqn:A}) is implied by
\begin{equation}\tag{B}
	\alpha \le n,\quad \alpha+\beta+\gamma \le n+c\implies R^{\alpha}(pr_S)_*\left(pr_X^*(\Omega_X^\beta\otimes(\mathcal B^{\otimes\gamma})^*)\otimes\mathcal I\right)=0.
\label{eqn:B}
\end{equation}
and
\begin{equation}\tag{C}
	\alpha \le n,\quad \alpha+\beta+\gamma \le n+c\implies H^{\alpha}\left(X,\Omega_X^\beta\otimes (\mathcal D^{\gamma})^*\right)=0
\label{eqn:C}
\end{equation}
%Maybe I want to elaborate on how we reduce to (B) and (C) more...
Recalling that $\mathcal I=pr_X^*\mathcal O_X(m)$ and applying Serre duality we see that (\ref{eqn:B}) and (\ref{eqn:C}) are implied by the following proposition, replacing $\mathcal F$ with $\Omega_X^{n+1-\beta}.$
\begin{proposition}
	Let $\mathcal F$ be a coherent sheaf on $X$ and $\gamma\ge 0$ an integer. Then there is constant $N=N(\mathcal F,\gamma)$ such that $m\ge N$ implies
	\begin{enumerate}
	    \item $H^i(X,\mathcal F\otimes\mathcal B^{\otimes\gamma}\otimes\mathcal O_X(m))=0$ for all $i>0,$
		\item $H^i(X,\mathcal F\otimes\mathcal D^\gamma)=0$ for all $i>0.$
	\end{enumerate}
\label{prop:N3}
\end{proposition}
To prove this proposition, we need two preliminary results. These require a version of Serre's vanishing theorem. The proofs of all three can be found in \cite{Nor93} (see also \cite{Ser55}).
\begin{theorem}
	Let $\mathcal G$ be a coherent sheaf on $X\times \mathbb P^{m_1}\times\cdots\times \mathbb P^{m_r}$ and let $p_j$ for $0\le j\le r$ denote the projections. Then
	\[
		R^i(p_0)_*\left(\mathcal G\otimes p_1^*\mathcal O(b_1)\otimes\cdots\otimes p_r^*\mathcal O(b_r)\right)=0
	\]
	for all $i>0$ as long as $\min\{b_1,\ldots,b_r\}$ is sufficiently large.
\label{thm:Serre}
\end{theorem}
\begin{corollary}
	Let $Z=X^{\gamma+1}$ and let $p_j:Z\to X$ for $0\le j\le \gamma$ denote the projections. Let $\mathcal G$ be a coherent sheaf on $Z$ and let
	\[
		\mathcal W = p_1^*\mathcal O_X(m)\otimes\cdots\otimes p_\gamma^*\mathcal O_X(m)
	\]
	Then for all $i>0$ and all $r>0$ we have
	\[
		H^i\left(X,\mathcal O_X(m)\otimes (p_0)_*(\mathcal G\otimes\mathcal W)\right)=0
	\]
	as long as $m$ is sufficiently large.
\label{cor:1}
\end{corollary}

\begin{corollary}
	With the notation of Corollary \ref{cor:1}, let $q=\prod_{j=1}^\gamma p_j:Z\to X^\gamma$ and assume that the restriction of $q$ to $\supp\mathcal G$ is a finite morphism. Then for all $i>0$ we have
	\[
		H^i\left(X, (p_0)_*(\mathcal G\otimes \mathcal W)\right)=0.
	\]
\label{cor:2}
\end{corollary}
%proof of the corollary
\iffalse
\begin{proof}
    Since the restriction of $q$ to $\supp\mathcal G$ is finite, we have the equality
    \[
        	H^i\left(X, (p_0)_*(\mathcal G\otimes \mathcal W)\right) = H^i\left(X^\gamma,q_*\mathcal G\otimes\bigotimes_{i=1}^\gamma p_i^*\mathcal O_X(m)\right).
    \]
    The latter vanishes by Theorem \ref{thm:Serre}.
\end{proof}
\fi

\indent With these results we can proceed with the proof of Proposition \ref{prop:N3}. We keep the notation of the corollaries. The goal is to find coherent sheaves $\mathcal Z^0$ and $\mathcal Z^1$ on $Z$ such that
\begin{enumerate}
	\item $\mathcal Z^0$ and $\mathcal Z^1$ are flat with respect to $p_0:Z\to X.$
	\item $(p_0)_*(\mathcal Z^0\otimes \mathcal W)=\mathcal B^{\otimes\gamma}$ and $(p_0)_*(\mathcal Z^1\otimes \mathcal W)=\mathcal D^\gamma.$
	\item The restriction of $q$ to the support of $\mathcal Z^1$ is a finite morphism.
\end{enumerate}
If we let $\mathcal G=\mathcal Z^i \otimes p_0^*\mathcal F$ and apply Corollary \ref{cor:1} when $i=0$ and Corollary \ref{cor:2} when $i=1$ we get the desired result. Hence it remains to find appropriate $\mathcal Z^0$ and $\mathcal Z^1.$\\
\indent Define a complex of sheaves $\mathcal C^\bullet$ on $Z$ in the following way. For each subset $F\subseteq \{1,\ldots,\gamma\}$ we let
\[
	\Delta(F)=\{z\in Z\mid p_0(z)=p_i(z)\text{ for all }i\in F\}.
\]
Take note that $q$ restricted to each $\Delta(F)$ is a finite morphism. Let $\mathcal O_{\Delta(F)}$ denote the structure sheaf of $\Delta(F)$ (or, more precisely, the pushforward to $Z$ of the structure sheaf of $\Delta(F)$). Letting $\mathbb C[F]$ denote the free vector space on $F,$ we then define
\[
	\mathcal C^q = \bigoplus_{|F|=q}\bigwedge\nolimits^{\! q}(\mathbb C[F])\otimes \mathcal O_{\Delta(F)}.
\]
\indent Now we define the coboundary maps. Suppose $F\subseteq\{1,\ldots,\gamma\}$ and $|F|=q.$ If $G=F\cup\{a\}$ with $a\notin F$ then $\mathbb C[F]\subseteq \mathbb C[G].$ We then get an isomorphism
\[
	L(F,a):\bigwedge\nolimits^{\! q}(\mathbb C[F])\to\bigwedge\nolimits^{\! q+1}(\mathbb C[G])
\]
given by $L(F,a)(\omega)=a\wedge\omega.$ Furthermore, the inclusion $\Delta(G)\subseteq \Delta(F)$ gives a map on the structure sheaves
\[
	R(F,a):\mathcal O_{\Delta(F)}\to\mathcal O_{\Delta(G)}.
\]
The coboundary map $\delta^q:\mathcal C^q\to\mathcal C^{q+1}$ is then defined on the summand $\bigwedge^q(\mathbb Z[
F])\otimes \mathcal O_{\Delta(F)}$ to be $\sum_{a\notin F}L(F,a)\otimes R(F,a).$ It can be checked that $\mathcal C^\bullet$ is exact in degrees $q>0.$ Now let $\mathcal Z^q=\ker\delta^q.$ We get short exact sequences $T^q$
\begin{equation}
	0\to\mathcal Z^q\to\mathcal C^q\to\mathcal Z^q\to 0
	\label{eqn:seq}
\end{equation}
and since each $\mathcal C^q$ is $p_0$-flat by construction, decreasing induction on $q$ shows that each $\mathcal Z^q$ is $p_0$-flat.\\
\indent Denote by $\mathcal K^\bullet$ the two term complex 
\[
	0\to\mathcal A\to\mathcal O_X(m)\to 0
\]
where $\mathcal A$ is in degree $0.$ Then $(p_0)_*(\mathcal C^\bullet\otimes\mathcal W)=(\mathcal K^\bullet)^{\otimes\gamma}.$ It follows that
\begin{align*}
	\mathcal H^q((p_0)_*(\mathcal C^\bullet\otimes W))=
	\begin{cases}
		0 & q>0,\\
		\mathcal H^0(\mathcal K^\bullet)^{\otimes\gamma}=\mathcal B^{\otimes\gamma} & q=0.
	\end{cases}
\end{align*}
We get left exact sequences obtained by tensoring the exact sequence in \ref{eqn:seq} with $\mathcal W$ and applying $(p_0)_*$:
\[
	0\to (p_0)_*(\mathcal Z^q\otimes\mathcal W)\to (p_0)_*(\mathcal C^q\otimes\mathcal W)\to (p_0)_*(\mathcal Z^{q+1}\otimes\mathcal W).
\]
For $q=0$ we find that
\begin{align*}
	(p_0)_*(\mathcal Z^0\otimes\mathcal W) &= \ker((p_0)_*(\mathcal C^0\otimes\mathcal W)\to(p_0)_*(\mathcal Z^1\otimes\mathcal W))\\
	&=\ker((p_0)_*(\mathcal C^0\otimes\mathcal W)\to(p_0)_*(\mathcal C^1\otimes\mathcal W))\\
	&= \mathcal H^0((p_0)_*(\mathcal C^\bullet\otimes W))\\
	&=\mathcal B^{\otimes\gamma}.
\end{align*}
For $q=1$ we get
\begin{align*}
	(p_0)_*(\mathcal Z^1\otimes\mathcal W) &= \ker((p_0)_*(\mathcal C^1\otimes\mathcal W)\to(p_0)_*(\mathcal Z^2\otimes\mathcal W))\\
	&= \ker((p_0)_*(\mathcal C^1\otimes\mathcal W)\to(p_0)_*(\mathcal C^2\otimes\mathcal W))\\
	&= \im((p_0)_*(\mathcal C^0\otimes\mathcal W)\to(p_0)_*(\mathcal C^1\otimes\mathcal W))\\
	&= \coker((p_0)_*(\mathcal Z^0\otimes\mathcal W)\to(p_0)_*(\mathcal C^0\otimes\mathcal W))\\
	&= \coker(\mathcal B^{\otimes\gamma}\to\mathcal A^{\otimes\gamma})\\
	&= \mathcal D^\gamma.
\end{align*}
Thus $\mathcal Z^0$ and $\mathcal Z^1$ have the desired properties $(1)$ and $(2).$ Since $q$ restricted to each $\Delta(F)$ is a finite morphism, this gives us $(3),$ completing the proof.\\

\indent As a remark, Nori uses Proposition \ref{prop:N2'} along with the following proposition to deduce Theorem \ref{thm:N}.
\begin{proposition}
	Let $i:V\to U$ be a closed immersion and $p:U\to S$ be a morphism, where $U,$ $V,$ and $S$ are smooth complex varieties. Let $\Omega_U$ and $\Omega_V$ denote the respective cotangent sheaves. Suppose that for some integer $n$ the restriction map $i^*:R^a(p_S)_*\Omega_U^b\to R^a(p_S)_*\Omega_V^b$ is an isomorphism for $a<n,$ $a+b<2n$ and injective for $a\le n$ and $a+b\le 2n.$ Then the restriction map
	\[
		i^*:H^k(U,\mathbb Q)\to H^k(V,\mathbb Q)
	\]
	is an isomorphism for $k<2n$ and injective for $k=2n.$
	\label{prop:N4}
\end{proposition}

%Incomplete
The idea of the proof of Proposition \ref{prop:N4} is to view the relative cohomology $H^{k}(U,V;\mathbb C)$ as the hypercohomology of the complex of relative forms $\Omega_{U,V}^\bullet.$ Nori then equips the hypercohomology with two filtrations, the Hodge filtration $F^\bullet$ of Deligne (see \cite{Del74III}) and the filtration $G^\bullet$ coming from the hypercohomology spectral sequence. It follows fairly easily from the definitions that
\[
    F^p\mathbb H^k(U,\Omega_{U,V}^\bullet)\subseteq G^p\mathbb H^k(U,\Omega_{U,V}^\bullet)
\]
for each $p.$ The hypothesis of the theorem allows one to study the spectral sequence and show that $G^p\mathbb H^{n+p}(U,\Omega_{U,V}^\bullet)=0$ for $p\le n.$ Hence much of the Hodge filtration vanishes and from this we can conclude.

\section{Background on D-Modules}\label{DMod}
In the next section we are going to use Nori's theorem and the decompostion theorem to get a result about the graded pieces of a certain D-module on $X.$ To that end we will in this section review some generalities about D-modules. A standard referece for this material is \cite{HTT08}.

We denote by $\mathcal D_Y$ the sheaf of differential operators on a smooth algebraic variety $Y$ of dimension $n.$ This will be a sheaf of noncommutative rings over $Y.$ The definition is as follows. First suppose that $A$ is a commutative ring. In the endomorphism ring $\End_k A$ of $A,$ identify $f\in A$ with the endomorphism $x\mapsto fx.$ We define the \textbf{ring of (algebraic) differential operators} $D(A)\subset\End_k A$ recursively by first setting
\[
    F_0D(A)=A
\]
and then for $j\ge 1$ setting
\[
    F_jD(A) = \left\{P\in\End_k A\mid [P,f]\in F_{j-1}D(A)\text{ for all }f\in A\right\}.
\]
We then make the definition
\[
    D(A)=\bigcup_{j=0}^\infty F_jD(A).
\]
\indent \textbf{The sheaf of (algebraic) differential operators on $Y$}, denoted by $\mathcal D_Y,$ is defined as the unique sheaf of $\mathcal O_Y$-modules such that for every open affine subset $U=\spec A$ of $Y,$ we have $\Gamma(U,\mathcal D_Y) = D(A).$ It comes with a filtration by coherent $\mathcal O_Y$ sheaves $F_j\mathcal D_Y$ given by gluing together the filtrations $F_j$ as defined above, and we have the equalities
\begin{align*}
    F_0\mathcal D_Y = \mathcal O_Y, \qquad F_j\mathcal D_Y/F_{j-1}\mathcal D_Y = \Sym^j\mathcal T_Y
\end{align*}
where $\mathcal T_Y$ denotes the tangent sheaf of $Y.$ We think of $F_j\mathcal D_Y$ as the ``differential operators of order $\le j.$'' We call a quasi-coherent $\mathcal O_Y$-module $\mathcal M$ a \textbf{(left) $D$-module on $Y$} if it also has a (left) action by $\mathcal D_Y.$ We say $\mathcal M$ is a \textbf{filtered $D$-module on $Y$} if it has a filtration by coherent $\mathcal O_Y$-modules $F_k\mathcal M$ such that for all $k\in\mathbb Z$ and $j\ge 0$ we have
\[
    F_j\mathcal D_Y\cdot F_k\mathcal M\subset F_{j+k}\mathcal M.
\]
\indent Given a $D$-module $\mathcal M,$ we can form its \textbf{de Rham complex}
\[
    \DR(\mathcal M) = \left[\mathcal M\to \Omega_Y^1\otimes \mathcal M \to \Omega_Y^2\otimes \mathcal M\to\cdots\to \Omega_Y^n\otimes \mathcal M\right]
\]
placed in degrees $-n,\ldots,0.$ The maps are the ones induced by the action of $\mathcal D_Y$ on $\mathcal M.$ We usually think of $\DR(\mathcal M)$ as living in the derived category $D_c(\mathbb C_Y)$ of sheaves of $\mathbb C$-vector spaces on $Y$ whose cohomology sheaves are constructible (i.e. they are locally constant with respect to some stratification of $Y$). When $\mathcal M$ is a filtered $D$-module there is an induced filtration on $\DR(\mathcal M)$ given by
\[
    F_j\DR(\mathcal M) = \left[F_j\mathcal M\to \Omega_Y^1\otimes F_{j+1}\mathcal M \to\cdots\to \Omega_Y^n\otimes F_{j+n}\mathcal M\right].
\]
\indent For the rest of this paper we will not be working with general filtered $D$-modules, but instead with Hodge modules. A \textbf{(mixed) Hodge module} is a filtered $D$-module that comes with another filtration, called the ``weight filtration,'' and which satisfies certain technical conditions. For the purposes of this paper, we will only give a few examples of the properties that Hodge modules satisfy. For a more comprehensive overview of the theory see \cite{Sch19}.\\
\indent Now let $f:Z\to Y$ be a morphism of smooth algebraic varieties and let $\mathcal M$ be a $D$-module on $Z.$ There is a pushforward of $\mathcal M,$ written $f_+\mathcal M,$ which is defined in the following way. First we tensor with the canonical sheaf $\omega_Z,$ then we take the derived tensor product with the \textbf{transfer module} $\mathcal D_{Z\to Y}:= \mathcal O_Z\otimes_{f^{-1}\mathcal O_Y} f^{-1}\mathcal D_Y.$ Next, we take the derived pushforward, and finally we apply the functor $\sHom_{\mathcal D_Y}(\omega_Y,-)$. So in full we have the definition
\[
    f_+\mathcal M := \sHom_{\mathcal D_Y}\left(\omega_Y,\mathbf R f_*\left(\left(\omega_Z\otimes_{\mathcal O_Z}\mathcal M\right)\otimes_{\mathcal D_Z}^{\mathbf L}\mathcal D_{Z\to Y}\right)\right)
\]%Not sure if this definition is entirely correct
The definition is so cumbersome because, for technical reasons, the pushforward is only easily defined for right $D$-modules, so we need to first convert $\mathcal M$ into a right $D$-module by tensoring with $\omega_Z$ and then convert back by applying the functor $\sHom_{\mathcal D_Y}\left(\omega_Y,-\right).$ However, as defined the pushforward has some convenient properties. For example, from \cite[Theorem 7.1.1]{HTT08} we have the formula
\[
    \DR(f_+\mathcal M) = \mathbf Rf_* \DR(\mathcal M).
\]
If $\mathcal M$ is a filtered $D$-module, then the natural filtrations on the transfer module and the dualizing modules $\omega_Y$ and $\omega_Z$ induce a filtration on $f_+\mathcal M,$ so the pushforward also makes sense for filtered $D$-modules. One reason for working with Hodge modules is that they interact nicely with certain functors such as the pushforward. If $\mathcal M$ is a Hodge module on $Z$ then we have the following interaction with taking graded pieces (see \cite[Theorem 6.8]{Sch19}):
\[
     \gr_p^F\DR(f_+\mathcal M) = \gr_p^F\mathbf Rf_* \DR(\mathcal M) = \mathbf Rf_*\gr_p^F\DR(\mathcal M).
\]
The fundamental theorem here is the \textbf{decomposition theorem} (see \cite[Corollary 3.7]{Sch19}):%Can I mention the history here and just cite a source? (Brosnan, Fang, Nie, Pearlstein). Or maybe just cite Schnell here. Also not sure which hypotheses are correct...
\begin{theorem}
    Suppose $\pi:Z\to Y$ is a projective morphism of smooth varieties. Then there is a (not necessarily unique) isomorphism in the bounded derived category of filtered $D$-modules on $Y:$
    \[
        \pi_+\mathcal O_Z \cong \bigoplus_i \mathcal H^i\pi_+\mathcal O_Z[-i].
    \]
\label{thm:decomp}
\end{theorem}
Applying the Riemann-Hilbert correspondence to this decomposition yields a decomposition in the category of perverse sheaves
\[
     \DR(\pi_+\mathcal O_Z) \cong \bigoplus_i \mathcal H^i \mathbf R\pi_*\DR(\mathcal O_Z)[-i].
\]
We will say more about this decomposition in the case $Z=\mathscr X$ and $Y=P$ in the next section. We now discuss two examples.
\begin{example}
    The structure sheaf $\mathcal O_Y$ is a filtered $D$-module with the action of differentiation and the trivial filtration given by $F_0\mathcal O_Y=\mathcal O_Y$ and $F_{-1}\mathcal O_Y=0.$ In this case we have
    \[
        \DR(\mathcal O_Y) = \left[\mathcal O_Y\to\Omega_Y^1\to\cdots\to\Omega_Y^n\right]
    \]
    where the maps are the usual exterior derivatives. Note that this is quasi-isomorphic to the (shifted) constant sheaf $\mathbb C[n].$ The filtration is then given by $F_j\DR(\mathcal O_Y)=\DR(\mathcal O_Y)$ for $j>0$ and
    \[
        F_j\DR(\mathcal O_Y) = \left[\Omega_Y^{-j}\to\cdots\to\Omega_Y^n\right]
    \]
    for $j=-n,\cdots,0.$ It follows that
    \[
        \gr_{-p}^F\DR(\mathcal O_Y) = \Omega_Y^p[n-p]
    \]
    for $p=0,\ldots,n,$ where $\Omega_Y^p[n-p]$ denotes the sheaf $\Omega_Y^p$ considered as a one-term complex shifted into degree $-n+p.$
\end{example}

\begin{example}
    If $\mathcal M=\mathcal H$ is a (pure) variation of Hodge structure (VHS) on $Y,$ then the Gauss-Manin connection induces the structure of a $D$-module on $\mathcal H.$ In fact, the condition of Griffiths transversality is exactly the condition that $\mathcal H$ is actually a filtered $D$-module. Indeed, if we define $F_p\mathcal H=F^{-p}\mathcal H,$ where $F^\bullet$ denotes the Hodge filtration, then Griffiths transversality exactly says that
    \[
        F_1\mathcal D_Y\cdot F_p\mathcal H\subset F_{p-1}\mathcal H
    \]
    for every $p.$ By induction we then have that $\mathcal H$ is a filtered $D$-module. The filtration on the de Rahm complex is given by
    \begin{align*}
        F_{-p}\DR(\mathcal H)&=\left[F_{-p}\mathcal H\to F_{-p+1}\mathcal H\otimes\Omega_Y^1\to\cdots\to F_{-p+n}\mathcal H\otimes\Omega_Y^n \right]\\
        &=\left[F^p\mathcal H\to F^{p-1}\mathcal H\otimes\Omega_Y^1\to\cdots\to F^{p-n}\mathcal H\otimes\Omega_Y^n \right]
    \end{align*}
    where the maps are given by the Gauss-Manin connection. Hence the graded pieces take the form
    \[
        \gr_{-p}^F\DR(\mathcal H) = \left[\mathcal H^{p,n-p}\to\mathcal H^{p-1,n-p+1}\otimes\Omega_Y^1\to\cdots\right].
    \]
    In particular if $\mathcal H$ is a trivial VHS then the differentials in the above complex are zero; thus the cohomology sheaves take the form
    \begin{align*}
        \mathcal H^i\gr_{-p}^F\DR(\mathcal H) = \mathcal H^{p-d-i,n-p+d+i}\otimes \Omega_Y^{d+i} = \gr^F_{-p+n+i}\mathcal H\otimes \Omega_Y^{d+i}.
    \end{align*}
    This formula will be handy for a calculation in the next section.
\end{example}

% Need to seriously edit this section (yay.)
\section{Main Theorem}\label{MainThm}
We now return to the setting from the introduction. Recall that $X$ is a projective variety, $\mathcal O_X(1)$ is a very ample line bundle on $X,$ $P=\mathbb P\left(H^0(X,\mathcal O_X(m))\right)$ is a projective space of dimension $d,$ and $\mathscr X \subset P\times X$ is the universal hypersurface. Replacing $Y$ with $P\times X$, pushing forward by $pr_P$, and using the Riemann-Hilbert correspondence, we have
\begin{align*}
	\mathbf R (pr_P)_*\Omega^b_{P\times X}[-d-n-1+b] &= \mathbf R (pr_P)_*\gr^F_{-b}\DR_{P\times X}(\mathcal O_{P\times X})\\
	&= \gr^F_{-b}\DR((pr_P)_+\mathcal O_{P\times X}),
\end{align*}
and similarly for the universal hypersurface $\mathscr X,$ where $b\in\mathbb Z$ is fixed. For each $-n\le i\le n$ define
\begin{align*}
	K^{n+i}(X)=
	\begin{cases}
		H^{n+i}(X)& i\le 0\\
		H^{n+2+i}(X)(1)& i>0
	\end{cases}.
\end{align*}
Note that each $K^{n+i}(x)$ is a Hodge structure of weight $n+i.$ In \cite{Sch12} it is shown that
\begin{align*}
	\DR((pr_P)_+\mathcal O_{P\times X}) = \bigoplus_{i}A_i[-i],\qquad \DR((\pi_P)_+\mathcal O_{\mathscr X}) = \DR(\mathcal M)\oplus\bigoplus_{i}B_i[-i]
\end{align*}
where $A_i=H^{n+1+i}(X)\otimes \mathbb C[d]$ and $B_i=K^{n+i}(X)\otimes \mathbb C[d].$ In particular, all of the summands are Hodge modules supported on all of $P,$ and all except $\mathcal M$ come from constant variations of Hodge structure on $P.$
\indent Let us give a brief word on $\mathcal M,$ the only nontrivial Hodge module appearing here. Recall that $P^{\mathrm{sm}}\subseteq P$ is the open set over which the fibers of $\pi_P$ are smooth. Here we have a $D_{P^{\mathrm{sm}}}$-module $\mathcal H_{\mathrm{van}}^n.$ This is the vector bundle whose fiber at $p\in P^{\mathrm{sm}}$ is equal to $H_{\mathrm{van}}^n(X_p,\mathbb C),$ the vanishing cohomology of the hypersurface $X_p.$ Recall that this is defined as the kernel of the Gysin morphism $H^n(X_p,\mathbb C)\to H^{n+2}(X,\mathbb C)$ induced by the inclusion $X_p\subset X.$ The $D_{P^{\mathrm{sm}}}$-module structure is given by the Gauss-Manin connection $\nabla:\mathcal H_{\mathrm{van}}^n\to \mathcal H_{\mathrm{van}}^n\otimes \Omega_{P^{\mathrm{sm}}}^1.$ $\mathcal M$ is then obtained as the minimal extension of $\mathcal H_{\mathrm{van}}^n$ over the singular locus of $\pi_P.$ It is a Hodge module on $P$ with filtration $F_\bullet$ induced by the Hodge filtration of $\mathcal H_{\mathrm{van}}^n.$ See \cite{Sch12} for more on minimal extensions of $D$-modules.\\
\indent Using Proposition \ref{prop:N2} we can extract some information about the cohomology sheaves of $\gr^F_{-b}\mathcal M.$ Let us first get a handle on the graded pieces of $\DR((pr_P)_+\mathcal O_{P\times X}).$ For $k\in \bb N$ we have

\begin{align*}
	&\mathcal H^{-d-n-1+k}\gr^F_{-b}\DR((pr_P)_+\mathcal O_{P\times X})\\
	&\cong \bigoplus_{i=-n-1}^{n+1} \mathcal H^{-d-n-1+k}\gr^F_{-b}\left(H^{n+1+i}(X)\otimes \DR(\mathcal O_P)[-i]\right)\\
	&= \bigoplus_{i=-n-1}^{n+1} \mathcal H^{-d-n-1+k-i}\gr^F_{-b}\left(H^{n+1+i}(X)\otimes \DR(\mathcal O_P)\right)\\
	&= \bigoplus_{i=-n-1}^{n+1} \gr^F_{-b-n-1+k-i}H^{n+1+i}(X)\otimes \Omega_P^{k-n-1-i}\\
	&= \bigoplus_{i=-n-1}^{n+1} H^{b+n+1-k+i,k-b}(X)\otimes \Omega_P^{k-n-1-i}\\
	&= \bigoplus_{j=0}^{k} H^{b-k+j,k-b}(X)\otimes \Omega_P^{k-j}.
\end{align*}

%Below is the calculation with $\mathbb C$'s. Just in case I made a mistake
\iffalse
\begin{align*}
	&\mathcal H^{-d-n-1+k}\gr^F_{-b}\DR((pr_P)_+\mathcal O_{P\times X})\\
	&= \bigoplus_{i=-n-1}^{n+1} \mathcal H^{-d-n-1+k}\gr^F_{-b}\left(H^{n+1+i}(X)\otimes \mathbb C[d-i]\right)\\
	&= \bigoplus_{i=-n-1}^{n+1} \mathcal H^{k-n-1-i}\gr^F_{-b}\left(H^{n+1+i}(X)\otimes \mathbb C\right)\\
	&= \bigoplus_{i=-n-1}^{n+1} \gr^F_{-b+k-n-1-i}H^{n+1+i}(X)\otimes \Omega_P^{k-n-1-i}\\
	&= \bigoplus_{i=-n-1}^{n+1} H^{b-k+i+n+1,k-b}(X)\otimes \Omega_P^{k-n-1-i}\\
	&= \bigoplus_{j=0}^{k} H^{b-k+j,k-b}(X)\otimes \Omega_P^{k-j}.
\end{align*}%The reindexing here is worse. I should probably not do it (or fix it later) so as to not cause problems down the line.
\fi

Note that the first isomorphism here comes from the decomposition theorem and is not canonical. The last equality is just the reindexing $j=i+n+1$ so the indices don't get too long. We will return to the original indexing later when it is convenient. We can perform a similar calculation for $\DR((\pi_P)_+\mathcal O_{\mathscr X}):$
\begin{align*}
	&\mathcal H^{-d-n-1+k}\gr^F_{-b}\DR((\pi_P)_+\mathcal O_{\mathscr X})\\
	&\cong \mathcal H^{-d-n-1+k}\gr^F_{-b}\DR(\mathcal M)\oplus \bigoplus_{i=-n-1}^{n+1} \mathcal H^{-d-n-1+k}\gr^F_{-b}\left(K^{n+1+i}(X)\otimes \DR(\mathcal O_P)[-i]\right)\\
	&= \mathcal H^{-d-n-1+k}\gr^F_{-b}\DR(\mathcal M)\oplus \bigoplus_{i=-n-1}^{n+1} \mathcal H^{-d-n-1+k-i}\gr^F_{-b}\left(K^{n+1+i}(X)\otimes \DR(\mathcal O_P)\right)\\
	&= \mathcal H^{-d-n-1+k}\gr^F_{-b}\DR(\mathcal M)\oplus \bigoplus_{i=-n-1}^{n+1} \gr^F_{-b-n-1+k-i}K^{n+1+i}(X)\otimes \Omega_P^{k-n-1-i}\\
	&= \mathcal H^{-d-n-1+k}\gr^F_{-b}\DR(\mathcal M)\oplus \bigoplus_{i=-n-1}^{n+1} K^{b+n+1-k+i,k-b}(X)\otimes \Omega_P^{k-n-1-i}\\
	&= \mathcal H^{-d-n-1+k}\gr^F_{-b}\DR(\mathcal M)\oplus \bigoplus_{j=0}^{k} K^{b-k+j,k-b}(X)\otimes \Omega_P^{k-j}.
\end{align*}
It should be noted that we are taking the convention that $H^{p,q}=K^{p,q}=0$ if $p$ or $q$ are negative.\\ % In what follows, I would really like to refer to DR(O_P) as Q_P^H[d] but I don't really know how to do this...
\indent Let us take some time to discuss what the restriction map between these de Rham complexes looks like. Suppose $Y$ is a projective variety and $D$ is a smooth divisor in $Y.$ Let $U=Y\setminus D$ and let $i:D\to Y$ and $j:U\to Y$ be the inclusions. For a Hodge module $\mathcal V$ on $Y$ we have a distinguished triangle in the derived category
\[
    j_! j^! \mathcal V\to \mathcal V\to i_* i^* \mathcal V\to \cdots.
\]
When $\mathcal V=\DR(\mathcal O_Y),$ then $i^*\mathcal V=\DR(\mathcal O_D)[1]$ since the former starts in degree $-\dim Y$ whereas the latter starts in degree $-\dim D=-\dim Y + 1.$ Applying this to our setting, we get a map
\[
    \DR_P(\mathcal O_{P\times X})\to\ell_*\DR_P(\mathcal O_{\mathscr X})[1].
\]
Pushing forward using $pr_P$ and taking graded pieces then yields the restriction map
\[
    \ell^*:\gr_{-b}^F\DR_P((pr_P)_+\mathcal O_{P\times X})\to\gr_{-b}^F\DR_P((\pi_P)_+\mathcal O_{\mathscr X})[1].
\]
Hence on the $(-d-n-1+k)$-th cohomology sheaf, we have restriction maps
\begin{align*}
	\ell^*:\bigoplus_{i=0}^{k} H^{b-k+i,k-b}(X)\otimes \Omega_P^{k-i} \rightarrow \mathcal H^{-d-n+k}\gr_{-b}^F \DR(\mathcal M)\oplus \bigoplus_{j=0}^{k} K^{b-k+j,k-b}(X)\otimes \Omega_P^{k-j}.
\end{align*}
Now we can state the main result.
\begin{corollary}
	Let $k$ and $b$ be integers such that $k<2n$ and $k-b<n.$ Then the restriction map $\ell^*:\DR_P(pr_+\mathcal O_{P\times X})\to\DR_P(\pi_+\mathcal O_{\mathscr X})[1]$ induces an isomorphism
\begin{equation*}
	\bigoplus_{j=n+1}^{k} L^{j-n-1}H_{\mathrm{prim}}^{b-k+n+1,k-b-j+n+1}(X)\otimes \Omega_P^{k-j} \cong \mathcal H^{-d-n+k}\gr_{-b}^F \DR(\mathcal M),
\end{equation*}
%Maybe I should reindex i=j-n-1 here...
Where $L$ denotes the cup product with $c_1(\mathcal O_X(m)).$ In particular, for $k<n+1$ and $b>k-n$ we have $\mathcal H^{-d-n+k}\gr_{-b}^F \DR(\mathcal M)=0.$
	\label{cor:D}
\end{corollary}
\begin{proof}
    The map $\ell^*$ decomposes into a sum of maps. For simplicity we will first take a look at the summands not involving $\mathcal M.$ The restriction map on the level of Hodge modules, using the decomposition theorem, takes the form
    \[
        \bigoplus_i A_i[-i]\to \bigoplus_j B_j[-j+1]
    \]
    where $A_i$ and $B_j$ are as above. But $\Hom(A_i[-i],B_j[-j+1])=\Ext^{i-j+1}(A_i,B_j)$ which is zero if $i < j-1.$ It follows that the summands of the map on the level of cohomology sheaves can only be nonzero if $i\ge j-1.$ Since Proposition \ref{prop:N2} guarantees that this map is actually an isomorphism for certain $k$ and $b,$ it must be that the ``diagonal'' summands where $i=j-1$ are actually isomorphisms in these cases. Thus on the level of graded pieces, the summands of the restriction map which do not involve $\mathcal M$ take the form
    \begin{align*}
	    \ell_j^*:H^{b-k+j,k-b}(X)\otimes \Omega_P^{k-j}\rightarrow K^{b-k+j,k-b}(X)\otimes \Omega_P^{k-j}.
    \end{align*}
    \indent Now, for $j\le n$ we have $K^{b-k+j,k-b}(X)=H^{b-k+j,k-b}(X)$ and $\ell_j^*$ is simply the identity map. For $j>n$ we instead have $K^{b-k+j,k-b}(X)=H^{b-k+j+1,k-b+1}(X)$ and $\ell_j^*$ is the map $L,$ i.e. it is cup product with $c_1(\mathcal O_X(m)).$ Nori's theorem tells us that the map $\ell^*$ above is an isomorphism for $k<2n$ and $k-b<n.$ Since the identity map is clearly an isomorphism, it follows that the map
    \begin{align*}
		    \bigoplus_{j=n+1}^{k} H^{b-k+j,k-b}(X)\otimes \Omega_P^{k-j} \rightarrow \mathcal H^{-d-n+k}\gr_{-b}^F \DR(\mathcal M)
    \end{align*}
    must be surjective and map $\ker L$ isomorphically onto the codomain. In other words, this map induces the desired isomorphism.
\end{proof}

%%%%Cor d went here

\section{Applications}\label{Apps}
\indent This calculation continues to hold when we restrict to $P^{\mathrm{sm}}.$ Here we have
\begin{align*}
	\DR(\mathcal H_{\mathrm{van}}^n) = \left[\mathcal H_{\mathrm{van}}^n\mathop{\rightarrow}^\nabla \mathcal H_{\mathrm{van}}^n\otimes \Omega_{P^{\mathrm{sm}}}^1\mathop{\rightarrow}^\nabla  \mathcal H_{\mathrm{van}}^n\otimes\Omega_{P^{\mathrm{sm}}}^2 \rightarrow \cdots \rightarrow  \mathcal H_{\mathrm{van}}^n\otimes\Omega_{P^{\mathrm{sm}}}^d\right]
\end{align*}
and consequently
\begin{align*}
	\gr_{-b}^F \DR(\mathcal H_{\mathrm{van}}^n) = \left[\mathcal H_{\mathrm{van}}^{b,n-b}\mathop{\rightarrow}^{\overline\nabla} \mathcal H_{\mathrm{van}}^{b-1,n-b+1}\otimes \Omega_{P^{\mathrm{sm}}}^1\mathop{\rightarrow}^{\overline\nabla} \cdots\right]
\end{align*}
where $\overline\nabla$ is the map induced by $\nabla$ using Griffith's transversality. The isomorphism of Corollary $\ref{cor:D}$ gives us an explicit description of the cohomology of this complex of sheaves on $P^{\mathrm{sm}}.$ In the particular case $k=n$ and $b>0$ observe that this tells us that the map $$\overline\nabla:\mathcal H_{\mathrm{van}}^{b,n-b}\rightarrow \mathcal H_{\mathrm{van}}^{b-1,n-b+1}\otimes \Omega_{P^{\mathrm{sm}}}^1$$ is injective. Explicitly this means the following:
\begin{corollary}
	Let $b>0$ and suppose that $\sigma$ is a section of $F^b\mathcal H_{\mathrm{van}}^n$ such that $\nabla\sigma$ lies in $F^{b-2}\mathcal H_{\mathrm{van}}^n\otimes\Omega_{P^{\mathrm{sm}}}^1.$ Then $\sigma$ lies in $F^{b-1}\mathcal H_{\mathrm{van}}^n.$
\end{corollary}

\indent Another consequence of Corollary \ref{cor:D}
\begin{corollary}
    For $n\ge 3$ we have
    \[
        \bb H^{-d+1}(P^{\mathrm{sm}},\DR( \mathcal H_{\mathrm{van}}^n)) = H^n_{\mathrm{prim}}(X).
    \]
    \label{cor:N}
\end{corollary}
\begin{proof}
    In the hypercohomology spectral sequence
    \begin{align*}
	    E_2^{p,q}=H^p(P^\mathrm{sm},\mathcal H^q\gr^F_{-b}\DR(\mathcal M)) \implies \bb H^{p+q}(P^\mathrm{sm},\gr^F_{-b}\DR(\mathcal M)),
    \end{align*}
    Corollary \ref{cor:D} shows us that for $b\ge 2$ the only nonzero $E_2$ term of degree $-d+1$ is $E_2^{0,-d+1}.$ By an argument from induction and the fact that $E_2^{\ell,-d+1-\ell+1}=0$ for each $\ell\ne 0$, we also see that the differentials $d_r$ to and from $E_r^{0,-d+1}$ are zero for all $r\ge 2.$ Thus $E_2^{0,-d+1}=E_{\infty}^{0,-d+1}$ and we have
    \begin{align*}
        \gr_{-b}^F\mathbb H^{-d+1}(P^\mathrm{sm},\DR(\mathcal M)) &= \mathbb H^{-d+1}(P^\mathrm{sm},\gr^F_{-b}\DR(\mathcal M))\\
        &= E_{\infty}^{0,-d+1}\\
        &= E_2^{0,-d+1}\\
        &= H^0(P^\mathrm{sm},\mathcal H^{-d+1}\gr_{-b}^F\DR(\mathcal M))\\
        &= H^{b,n+1-b}_{\mathrm{prim}}(X).
    \end{align*}
    Since $n\ge 3,$ and the morphisms are the graded pieces of a morphism of Hodge structures
    \[
         \bb H^{-d+1}(P^{\mathrm{sm}},\DR( \mathcal H_{\mathrm{van}}^n)) = H^n_{\mathrm{prim}}(X),
    \]
    the conjugation symmetry of the Hodge decomposition allows us to conclude that it is actually an isomorphism on all the graded pieces, thus the whole map of Hodge structures is an isomorphism.
\end{proof}

%Not sure how correct this is
Let's take a closer look at the isomorphism
\[
	H^{b,n+1-b}_{\mathrm{prim}}(X)\otimes \mathcal O_{P^{\mathrm{sm}}}\cong \mathcal H^{-d+1}\gr_{-b}^F\DR(\mathcal H_{\mathrm{van}}^n)
\]
arising from Corollary \ref{cor:D} in the case $k=n+1.$ In what follows let $L$ denote the Leray filtration on $\Omega_{\mathscr{X}}^b$ for the map $\pi$ given by
\[
	L^i\Omega_{\mathscr{X}}^b = \Omega_{\mathscr{X}}^{b-i}\wedge \pi^*\Omega_{P^{\mathrm{sm}}}^i.
\]
This filtration yields the Leray spectral sequence
\[
	E_2^{p,q}=R^{p+q}\pi_*\gr_L^p\Omega_{\mathscr X}^b\implies R^{p+q}\pi_*\Omega_{\mathscr X}^b.
\]
It can be shown (see \cite[Proposition 5.9]{Voi03II}) that the complexes
\[
	\mathcal K_{b,q} := \gr_{-b}^F \DR(\mathcal H^{b+q})[-d] = \mathcal H^{b,q} \mathop{\rightarrow}^{\overline\nabla} \mathcal H^{b-1,q+1}\otimes\Omega_{P^{\mathrm{sm}}}^1 \mathop{\rightarrow}^{\overline\nabla}\cdots
\]
can be identified with the first page of this spectral sequence. Explicitly we have $\mathcal K_{b,q}^p = E_1^{p,q}.$\\
\indent Consider a primitive class $\alpha\in H^{b,n+1-b}_{\mathrm{prim}}(X) \subseteq H^{n+1-b}(X,\Omega_X^b).$ Then $\alpha$ admits a pullback $\beta=\pi_X^*\alpha\in H^{n+1-b}(\mathscr X,\Omega_{\mathscr X}^b)$ which restricts to $0$ on each smooth hyperplane section of $X.$ In other words,
\[
	\beta\in\ker\left(H^{n+1-b}(\mathscr X,\Omega_{\mathscr X}^b)\to H^0(P^{\mathrm{sm}},R^{n+1-b}\pi_*\Omega_{\mathscr X/P^{\mathrm{sm}}}^b)\right)
\]
Consider now the image of $\beta$ in $H^0(P^{\mathrm{sm}},R^{n+1-b}\pi_*\Omega_{\mathscr X}^b),$ which we also denote by $\beta.$ Using the exact sequence
\[
	0\to L^1\Omega_{\mathscr X}^b\to \Omega_{\mathscr X}^b\to \Omega_{\mathscr X/P^{\mathrm{sm}}}^b\to 0
\]
we see that
\begin{align*}
	L^1R^{n+1-b}\pi_*\Omega_{\mathscr X}^b &= \im\left(R^{n+1-b}\pi_*L^1\Omega_{\mathscr X}^b\to R^{n+1-b}\pi_*\Omega_{\mathscr X}^b\right)\\
	&= \ker\left(R^{n+1-b}\pi_*\Omega_{\mathscr X}^b\to R^{n+1-b}\pi_*\Omega_{\mathscr X/P^{\mathrm{sm}}}^b\right).
\end{align*}
Hence $\beta$ actually lies in $H^0(P^{\mathrm{sm}},L^1R^{n+1-b}\pi_*\Omega_{\mathscr X}^b)$ and has an image in the vector space $H^0(P^{\mathrm{sm}},\gr_L^1R^{n+1-b}\pi_*\Omega_{\mathscr X}^b).$ For degree reasons, in the Leray spectral sequence we have 
\[
	\gr_L^1R^{n+1-b}\pi_*\Omega_{\mathscr X}^b = E_\infty^{1,n-b} \subseteq E_2^{1,n-b},
\]
and since

\iffalse
\[
	E_2^{1,n-b} = \frac{\ker(\overline\nabla:\mathcal H^{b,n-b}\to\mathcal H^{b-1,n-b+1}\otimes \Omega_{P^{\mathrm{sm}}}^1)}{\im(\overline\nabla:\mathcal H^{b-1,n-b+1}\otimes \Omega_{P^{\mathrm{sm}}}^1\to\mathcal H^{b-2,n-b+2}\otimes \Omega_{P^{\mathrm{sm}}}^2)}
\]
\fi

\[
    E_2^{1,n-b} = \mathcal H^1(E_1^{\bullet, n-b}) = \mathcal H^1(\mathcal K_{b,n-b}^\bullet) = \mathcal H^{-d+1}\gr_{-b}^F\DR(\mathcal H_{\mathrm{van}}^n)
\]
we get an image of $\beta$ in the global sections of $\mathcal H^{-d+1}\gr_{-b}^F\DR(\mathcal H_{\mathrm{van}}^n).$ Putting this all together yields a map
\[
	H^{b,n+1-b}_{\mathrm{prim}}(X)\otimes \mathcal O_{P^{\mathrm{sm}}}\to \mathcal H^{-d+1}\gr_{-b}^F\DR(\mathcal H_{\mathrm{van}}^n).
\]
It seems likely that this is the map of Corollary \ref{cor:D}, however I have not been able to check that these maps coincide.

\printbibliography
\end{document}